\documentclass{amsart}
\usepackage[all]{xy}
\usepackage{amsmath} 
\usepackage{float}
\usepackage{amsfonts}
\usepackage{amssymb}
\usepackage{graphics}
\usepackage{latexsym}
\date{}

\newtheorem{theo}{Theorem}
\newtheorem{lem}[theo]{Lemma}
\newtheorem{prop}[theo]{Proposition}

\newtheorem{nota}[theo]{Notation}

\newcommand{\beq}{\begin{equation}}
\newcommand{\eeq}{\end{equation}}

\newcommand{\h}{{\mathbb{H}}}
\newcommand{\Z}{{\mathbb{Z}}}

\newcommand{\Ker}{\rm Ker\,}

\renewcommand{\mod}{{\rm mod}}

\begin{document}

\title{Reidemeister-Schreier's algorithm for 2-coverings of Seifert manifolds}

\footnotetext{\noindent{Mathematics Subject Classification 2010 : Primary 57M10, 
Secondary 57M05, 57S25}}

\author{A. Bauval, C. Hayat} 

\maketitle

\begin{abstract}

\noindent 
It is classical that given any Seifert structure on $N$,  Reidemeister-Schreier's algorithm  produces a presentation of all index 2 subgroups   of $\pi_1(N)$, described   as the fundamental group of some Seifert manifolds. The new result of this article is  concise  formulas that gather all possible cases.
%\emph{.}
\end{abstract}

\markboth{RS's algorithm for 2-coverings of Seifert manifolds}{RS's algorithm for 2-coverings of Seifert manifolds}\markright{RS's algorithm for 2-coverings of Seifert manifolds}

\section{Introduction}\label{intro}
The index 2 subgroups  of $\pi_1(N)$ are  kernel of epimorphisms $\varphi : \pi_1(N) \to \Z_2.$ When $N$ is a Seifert manifold (described by its Seifert invariants) and one needs  a list of all its 2-coverings, it is necessary to explicit a combinatoric way to gather together all the data. Theorems 1 and 3 give  Reidemeister-Schreier concise answers, \cite{MKS}, \cite{zvc}.

The  notations are given in the section after this introduction. Section 3 studies the situation where the morphism $\varphi$ maps the generator corresponding to the regular fiber to 1. This is Theorem 1. Theorem 3 stating the result when $\varphi$ maps the generator corresponding to the regular fiber to $0$ is expressed in the fourth section. The following subsections prove this theorem. In the first subsection, a crucial lemma (Lemma \ref{lem:killvj}) proves that  if two morphisms from $\pi_1(N)$ to $\Z_2$ map $m\geq 0$ exceptional fibres to 1 and all the other generators to $0$ then their kernels are isomorphic. This gives  importance to Theorem 11  which explicits the kernel with these hypothesis. The second subsection studies the situation where $\varphi$ maps some generators corresponding to the basis  to $1$ and all the other generators to $0$. Theorem 3 is proved.

Each index 2 subgroup  of $\pi_1(N)$ is the fundamental group of a Seifert manifold $M$ and an associated  free involution $\tau$. The motivation of this study came to us through the study of Borsuk-Ulam type theorem for $(M,\tau)$ \cite{mat}, \cite{BGHZ}.

\section{Seifert invariants for the kernel}\label{RSS}

$N$ is {\sl any} Seifert manifold (orientable or not), as introduced in \cite{s}.

Following the notations of  \cite{orlik}, from now on, $N$ will be a Seifert manifold described by a list of Seifert invariants  $$\{e; (\epsilon, g);(a_1,b_1),\ldots,(a_n,b_n)\}.$$
We do not need them to be ``normalized'' (as defined in \cite{s} or \cite{orlik}): we only assume that $e$ is an integer, the type $\epsilon$ is detailed below, $g$ is the genus of the base surface, and for each $k$, the integers $a_k,b_k$ are coprime and $a_k\ne 0$.

Such invariants give the following presentation of the fundamental group of $N$:
$$\pi_1(N)=\left<{\begin{matrix}s_1,\ldots,s_n\\v_1,\ldots,v_{g'}\\h\end{matrix}\left|
\begin{matrix}
 [s_k,h]\quad\text{and}\quad s_k^{a_k}h^{b_k},&1\le k\le n\\
v_jhv_j^{-1}h^{-\varepsilon_j},&1\le j \le g'\\
h^{-e}s_1\ldots s_nV&\end{matrix}\right.}\right>.$$

\begin{itemize}
\item The type $\epsilon$ of $N$ equals:
	\begin{itemize}
	\item[$o_1$] if both the base surface and the total space are orientable  (which forces all $\varepsilon_j$'s to equal $1$);
	\item[$o_2$] if the base is orientable and the total space is non-orientable, hence $g\ge 1$ (all $\varepsilon_j$'s are assumed to equal $-1$);
	\item[$n_1$] if both the base and the total space are non-orientable (hence $g\ge 1$) and moreover, all $\varepsilon_j$'s equal $1$;
	\item[$n_2$] if the base is non-orientable (hence $g\ge 1$) and the total space is orientable (which forces all $\varepsilon_j$'s to equal $-1$);
	\item[$n_3$] if both the base and the total space are non-orientable and moreover, all $\varepsilon_j$'s equal $-1$ except $\varepsilon_1=1$, and $g\ge 2$;
	\item[$n_4$] if both the base and the total space are non-orientable and moreover, all $\varepsilon_j$'s equal $-1$ except $\varepsilon_1=\varepsilon_2=1$, and $g\ge 3$.
	\end{itemize}
\item The orientability of the base and its genus $g$ determine the number $g'$ of the generators $v_j$'s and the word $V$ in the last relator of $\pi_1(N)$:
	\begin{itemize}
	\item when the base is orientable, $g'=2g$ and  $V=[v_1,v_2]\ldots[v_{2g-1},v_{2g}]$;
	\item when the base is non-orientable, $g'=g$ and  $V=v_1^2\ldots v_g^2$.
	\end{itemize}
\item The generator $h$ corresponds to the generic regular fibre.
\item The generators $s_k$ for $1\le k\le n$ correspond to (possibly) exceptional fibres.
\end{itemize}

The subgroups of index 2 of $\pi_1(N)$ are the kernel of epimorphism $\varphi : \pi_1(N) \to \Z_2.$ The two next subsections describe $\Ker(\varphi)$ as the fundamental group of some Seifert manifold given by a similar list of invariants, when $\varphi(h)=1$ (Theorem \ref{theo:phi(h)=1}) and when $\varphi(h)=0$ (Theorem \ref{theo:phi(h)=0}).

\section{If $\varphi$ maps $h$ to $1$}\label{phi(h)=1}

\begin{theo}\label{theo:phi(h)=1}If $\varphi$ maps $h$ to $1$ then its kernel is the fundamental group of the Seifert manifold given by the following invariants:
$$\{\frac{e-m}2-m'; (\epsilon, g);(a_1,b'_1  ),\ldots,(a_n,b'_n)\},$$
where $b'_k=\begin{cases}\frac{b_k}2&\text{if $b_k$ is even}\\\frac{a_k+b_k}2&\text{if $b_k$ is odd}\end{cases}$, $m$ is the number of odd $b_k$'s, and $$\begin{cases}m'=0&\text{if $\epsilon=o_1,n_2$}\\m'\equiv\sum\varphi(v_j) (\mod~2)&\text{if $\epsilon=o_2,n_1$}\\m'\equiv\varphi(v_1) (\mod~2)&\text{if $\epsilon=n_3$}\\m'\equiv\varphi(v_1)+\varphi(v_2) (\mod~2)&\text{if $\epsilon=n_4$}.\end{cases}$$\end{theo}

Note that in the non-orientable cases, $m'$ is only determined modulo $2$, which is sufficient to determine the Seifert manifold.

\begin{proof}Necessarily, all $a_k$'s are odd, $\varphi(s_k)=b_k(\mod~2)$, and $e+m$ is even. Let us choose a presentation of $\pi_1(N)$ adapted to $\varphi$ by keeping $h$ untouched but taking new generators $s'_k$, $v'_j$ mapped to $0$ by $\varphi$:
$$s'_k=\begin{cases}s_k&\text{if $b_k$ is even}\\h^{-1}s_k&\text{if $b_k$ is odd}\end{cases}\qquad v'_k=\begin{cases}v_k&\text{if $\varphi(v_k)=0$}\\h^{-1}v_k&\text{if $\varphi(v_k)=1$.}\end{cases}$$
The new presentation of $\pi_1(N)$ corresponds to the Seifert invariants
$$\{e-m-2m'; (\epsilon, g);(a_1,2b'_1),\ldots,(a_n,2b'_n)\},$$
where the $b'_k$'s and $m$ are as stated, and $$m'=\begin{cases}0&\text{if $\epsilon=o_1,n_2$}\\
\#\{j~\text{odd}\mid\varphi(v_j)=1\}-\#\{j~\text{even}\mid\varphi(v_j)=1\}&\text{if $\epsilon=o_2$}\\
\#\{j\mid \varphi(v_j)=1~\text{and}~\varepsilon_j=1\}&\text{if $\epsilon=n_1,n_2,n_3,n_4$}\end{cases}$$
(hence $m'$ fullfills the condition of the statement).

Choosing $q=h$, Reidemeister-Schreier's algorithm produces a presentation of $\Ker(\varphi)$ with
\begin{itemize}
\item generators:
	\begin{itemize}
	\item for $1\le k\le n$, $(y_k,y'_k)=(s'_k,qs'_kq^{-1})$
	\item for $1\le j\le g'$, $(x_j,x'_j)=(v'_j,qv'_jq^{-1})$
	\item $(z,z')=(hq^{-1},qh)$
	\end{itemize}
\item relations:
	\begin{itemize}
	\item $z=1$
	\item  for $1\le k\le n$, $y'_k=y_k$, $[y_k,z']=1$ and $y_k^{a_k}z'^{b'_k}=1$
	\item for $1\le j\le g'$, $x_jz'x_j^{-1}z'^{-\varepsilon_j}=1$ and $x'_j=\begin{cases}x_j&\text{if $\varepsilon_j=1$}\\z'x_j&\text{if $\varepsilon_j=-1$}\end{cases}$
	\item $y_1\ldots y_nW=z'^{(e+m+2m')/2}$, with $$W=\begin{cases}[x_1,x_2]\ldots[x_{2g-1},x_{2g}]&\text{if $\epsilon=o_1,o_2$}\\x_1^2\ldots x_g^2&\text{if $\epsilon=n_1,n_2,n_3,n_4$.}\end{cases}$$
	\end{itemize}
\end{itemize}
Eliminating the redundant generators $z,y'_k,x'_k$ yields the result.
\end{proof}

\section{If $\varphi$ maps $h$ to $0$}\label{phi(h)=0}
\subsubsection{Results of the two next subsections}
Denote by $m$ the number of $s_k$'s mapped to $1$ by $\varphi$ and assume (if $m>0$) that these $m$ $s_k$'s are the first ones. (This reordering of the $s_k$'s may be achieved by an obvious change of presentation of $\pi_1(N)$, using repeatedly the equation $ss'=(ss's^{-1})s$.) The next theorem announces the conclusion of Theorems \ref{theo:s1smvj} and \ref{theo:vj}, which will be proved in the two next subsections. The following notations will be used to state the results.

\begin{nota}\label{nota:FOCFm}
The notations $F_{OC}$ and $F_m$ will respectively denote
$$F_{OC}=(a_1,b_1),(a_1,-b_1),(a_2,b_2),(a_2,-b_2),\ldots,(a_n,b_n),(a_n,-b_n)$$
$$F_m=(a_1/2,b_1),(a_2/2,b_2),\ldots ,(a_m/2,b_m),$$$$(a_{m+1},b_{m+1}),(a_{m+1},b_{m+1}),(a_{m+2},b_{m+2}),(a_{m+2},b_{m+2}),\ldots,(a_n,b_n),(a_n,b_n)$$
\end{nota}

\begin{theo}\label{theo:phi(h)=0}If $\varphi(h)=0$, denoting by $m$ the number of $s_k$'s mapped to $1$ by $\varphi$ and assuming these are the first ones,  $\Ker(\varphi)$ is the fundamental group of the Seifert manifold given by the following invariants:
\begin{itemize}
\item (Orientation covers) If $m=0$ and
	\begin{itemize}
	\item if $\epsilon=o_2$ and $\varphi$ maps all $v_j$'s to $1$: $\{0;(o_1,2g-1);F_{OC}\}$
	\item if $\epsilon=n_1$ and $\varphi$ maps all $v_j$'s to $1$: $\{0;(o_1,g-1);F_{OC}\}$
	\item if $\epsilon=n_3$ and $\varphi$ sends only $v_1$ to $1$, or if $\epsilon=n_4$ and $\varphi$ sends only $v_1,v_2$ to $1$: $\{0;(n_2,2g-2);F_{OC}\}$
	\end{itemize}
\item (Exotic cases) If $m=0$ and $\epsilon=n_2,n_3,n_4$ and $\varphi$ maps all $v_j$'s to $1$:
	\begin{itemize}
	\item if $\epsilon=n_2$: $\{2e; (o_1, g-1);F_0\}$
	\item if $\epsilon=n_3,n_4$: $\{0; (o_2, g-1);F_0\}$
	\end{itemize}
\item (Ordinary cases) In all other cases: $\{e';(\epsilon',G);F_m\}$ with
$$e'=\begin{cases}2e&\text{if $\epsilon=o_1,n_2$}\\0&\text{if $\epsilon=o_2,n_1,n_3,n_4$}\end{cases}\qquad\epsilon'=\begin{cases}\epsilon&\text{if $\epsilon=o_1,o_2,n_1,n_2,n_4$}\\n_4&\text{if $\epsilon=n_3$}\end{cases}$$$$\text{and}~G=\begin{cases}\frac m2-1+2g&\text{if $\epsilon=o_1,o_2$}\\m-2+2g&\text{if $\epsilon=n_1,n_2,n_3,n_4$.}\end{cases}$$
\end{itemize}
\end{theo}

\subsubsection{If $\varphi$ maps $h$ to $0$ but maps some $s_k$'s to $1$.}\label{m>0}

Later on (Lemma \ref{lem:killvj}), we shall reorder the $v_j$'s in the same spirit as we did for the $s_k$'s, and show that the isomorphism type of $\Ker(\varphi)$ is in fact independent from the values of $\varphi$ on the $v_j$'s, which reduces the computation of $\Ker(\varphi)$ to the particular case where $\varphi$ vanishes on all $v_j$'s. But before performing such a reduction, we need to show that in that particular case, $\Ker(\varphi)$ is the fundamental group a non-orientable Seifert manifold whenever $N$ is non-orientable.

So, let us first compute $\Ker\varphi$ in the particular case where $\varphi$ vanishes on all $v_j$'s. The following lemma is an intermediate step for this computation: it gives a presentation of $\Ker\varphi$ where the exceptional fibers gently appear, but where the long relation $W$ and the  $\pm$  signs may still be of a ``hybrid'' form.

\begin{lem}\label{lem:hybrid}If $\varphi$ maps $s_1,\ldots,s_m$ to $1$ and all other generators to $0$ then a presentation of its kernel is:
$$\Ker(\varphi)=\left<{\begin{matrix}s'_1,\ldots,s'_{n'}  \\v'_1,\ldots,v'_{g''}\\z\end{matrix}\left|
\begin{matrix}
 [s'_k,z]\quad\text{and}\quad s_k'^{a'_k}z^{b'_k},&1\le k\le n'\\
v'_jzv_j'^{-1}z^{-\varepsilon'_j},&1\le j \le g'' \\
z^{-2e}s'_1\ldots s'_{n'}W&\end{matrix}\right.}\right>,$$
where
\begin{itemize}
\item
	\begin{itemize}
	\item $n'=m+2(n-m)$,
	\item $(a'_k,b'_k)=(a_k/2,b_k)$ for $k\le m$,
	\item $(a'_k,b'_k)=(a'_{k+n-m},b'_{k+n-m})=(a_k,b_k)$ for $m<k\le n$,
	\end{itemize}
\item
	\begin{itemize}
	\item $g''=(m-2)+2g'$
	\item $\varepsilon'_j=1$ for $j\le m-2$,
	\item$\varepsilon'_{m-2+j}=\varepsilon'_{m-2+g'+j}=\varepsilon_j$ for $1\le j\le g'$,
	\end{itemize}
\item $W=\begin{cases}[v'_1,v'_2]\ldots[v'_{g''-1},v'_{g''}]&\text{if $\epsilon=o_1,o_2$,}\\
[v'_1,v'_2]\ldots[v'_{m-3},v'_{m-2}]v_{m-1}'^2\ldots v_{g''}'^2&\text{if $\epsilon=n_1,n_2,n_3,n_4$.}\end{cases}$
\end{itemize}
\end{lem}

\begin{proof}
Choosing $q=s_1$, Reidemeister-Schreier's algorithm produces a presentation of $\Ker(\varphi)$ with
\begin{itemize}
\item generators:
	\begin{itemize}
	\item for $1\le k\le n$, $(y_k,y'_k)=\begin{cases}(s_kq^{-1},qs_k)&\text{if $k\le m$}\\(s_k,qs_kq^{-1})&\text{if $k>m$}\end{cases}$
	\item for $1\le j\le g'$, $(x_j,x'_j)=(v_j,qv_jq^{-1})$
	\item $(z,z')=(h,qhq^{-1})$.
	\end{itemize}
\item relations:
	\begin{itemize}
	\item $y_1=1$ and $z'=z$
	\item $\forall k=1,\ldots,n$,
		\begin{itemize}
		\item $[y_k,z]=[y'_k,z]=1$
		\item $(y_ky'_k)^{a_k/2}z^{b_k}=1$ if $k\le m$
		\item $y_k^{a_k}z^{b_k}=y_k'^{a_k}z^{b_k}=1$ if $k>m$
		\end{itemize}	
	\item $x_jzx_j^{-1}z^{-\varepsilon_j}=x'_jzx_j'^{-1}z^{-\varepsilon_j}=1$ ($\forall j=1,\ldots,g'$)
	\item (I) $y'_2BCX=z^e$, where $B=y_3y'_4y_5y'_6\ldots y_{m-1}y'_m$, $C=y_{m+1}\ldots y_n$ and $$X=\begin{cases}[x_1,x_2]\ldots[x_{2g-1},x_{2g}]& \text{if $\epsilon=o_1,o_2$}\\x_1^2\ldots x_g^2&\text{if $\epsilon=n_1,n_2,n_3,n_4$}\end{cases}$$
	\item (II) $y'_1y_2y'_3y_4\ldots y'_{m-1}y_mC'X'=z^e$, where $C'=y'_{m+1}\ldots y'_n$ and $$X'=\begin{cases}[x'_1,x'_2]\ldots[x'_{2g-1},x'_{2g}]&\text{if $\epsilon=o_1,o_2$}\\x_1'^2\ldots x_g'^2&\text{if $\epsilon=n_1,n_2,n_3,n_4$.}\end{cases}$$
	\end{itemize}
\end{itemize}
(Note that $X$ and $X'$ allways commute with $z$.)
Let us make a change of generators in this presentation by suppressing $y'_1,y_2,y'_3,y_4,\ldots,y'_{m-1},y_m$ and introducing instead new generators $s'_1,\ldots,s'_m$ defined by:
$$s'_k=\begin{cases}y_2'^{-1}y_3^{-1}y_4'^{-1}y_5^{-1}\ldots y_{k-1}'^{-1}(y'_ky_k)y'_{k-1} \ldots y_5y'_4y_3y'_2&\text{if $k$ is odd}\\
y_2'^{-1}y_3^{-1}y_4'^{-1}y_5^{-1}\ldots y_{k-1}^{-1}(y_ky'_k)y_{k-1} \ldots y_5y'_4y_3y'_2&\text{if $k$ is even}\end{cases}$$
(in particular, $s'_1=y'_1$ and $s'_2=y_2y'_2$).
These new generators still commute with $z$, the relations  $(y_ky'_k)^{a_k/2}z^{b_k}=1$ become $s_k'^{a_k/2}z^{b_k}=1$, and the relation (II) becomes (III): $Ay_2'^{-1}B'C'X'=z^e$, where $A=s'_1\ldots s'_m$ and $B'=y_3^{-1}y_4'^{-1}y_5^{-1}y_6'^{-1}\ldots y_{m-1}^{-1}y_m'^{-1}$. Using (I) to eliminate the generator $y'_2$ from (III), which will become (IV) below, we are left with the following new presentation of $\Ker(\varphi)$:

\begin{itemize}
\item generators:
	\begin{itemize}
	\item $s'_k$ for $1\le k\le m$
	\item $y_3,y'_4,y_5,y'_6,\ldots,y_{m-1},y'_m$
	\item $y_k,y'_k$ for $m<k\le n$
	\item $x_j,x'_j$ for $1\le j\le g'$
	\item $z$
	\end{itemize}
\item relations:
	\begin{itemize}
	\item $[s'_k,z]=1$ and $s_k'^{a_k/2}z^{b_k}=1$ for $1\le k\le m$
	\item $[y_k,z]=[y'_k,z]=1$ and $y_k^{a_k}z^{b_k}=y_k'^{a_k}z^{b_k}$ for $m<k\le n$
	\item $[y_3,z]=\ldots=[y'_m,z]=1$
	\item $x_jzx_j^{-1}z^{-\varepsilon_j}=x'_jzx_j'^{-1}z^{-\varepsilon_j}=1$ for $1\le j\le g'$
	\item (IV) $ABCXB'C'X'=z^{2e}$.
	\end{itemize}
\end{itemize}

This last relation (IV) may be reordered by a new change of generators (replacing some of the generators by conjugates thereof) so as to become $ACC'BXB'X'=z^{2e}$, which we rewrite $ACC'(BXB'X^{-1})XX'=z^{2e}$. The parenthesis $BXB'X^{-1}$ is the product of $m-1$ elements, followed by the product (in the same order) of their inverses. It can be transformed into a product of $\frac m 2-1$ commutators, by another change of generators given by Lemma \ref{lem:prodcommut} below.

Provided that next (classical) lemma, this concludes the proof of Lemma \ref{lem:hybrid}, up to a renaming of the generators.
\end{proof}

\begin{lem}\label{lem:prodcommut}Let $F_{2k+1}$ be the free group over $g_0,\ldots,g_{2k}$. There exist elements $h_0,\ldots,h_{2k-1}\in F_{2k+1}$ such that:
$$g_0g_1\ldots g_{2k}g_0^{-1}g_1^{-1}\ldots g_{2k}^{-1}=[h_0,h_1][h_2,h_3]\ldots[h_{2k-2},h_{2k-1}]$$
and $F_{2k+1}$ is the free group over $h_0,\ldots,h_{2k-1},g_{2k}$.
\end{lem}

\begin{proof} Let $U_i=g_{2i}\ldots g_{2k}$ and $V_i=g_{2i}^{-1}\ldots g_{2k}^{-1}$. Then, $$U_kV_k=1\quad\text{and}\quad U_iV_i=[g_{2i}g_{2i+1},U_{i+1}g_{2i}^{-1}]U_{i+1}V_{i+1}.$$\end{proof}

Lemma \ref{lem:hybrid} produces a standard presentation of $\Ker(\varphi)$ only when $\epsilon=o_1$, or when $m=2$ and $\epsilon\ne n_4$. In other cases, the following Lemmas \ref{lem:commutsquar}, \ref{lem:epsiloncommut} and \ref{lem:epsilonsquar} tell how to normalize both $W$ (which must not be a mixture of commutators and squares) and the list of the $\varepsilon'_j$'s (for which the numbers of $1$'s and $-1$'s are partially prescribed).

\begin{lem}\label{lem:commutsquar}Let $F_3$ be the free group over $x,y,z$ and $\varepsilon\colon F_3\to\{1,-1\}$ the morphism defined by
$$\varepsilon(x)=\varepsilon(y)=1\qquad\text{and}\qquad\varepsilon(z)=-1.$$
There exist elements $u,v,w$ such that
$$\varepsilon(u)=\varepsilon(v)=\varepsilon(w)=-1\quad\text{and}\quad[x,y]z^2=u^2v^2w^2$$
and $F_3$ is the free group over $u,v,w$.
\end{lem}

\begin{proof}Take for instance $ u=xz$, $v=(zxz)^{-1}yz$ and $w=(yz)^{-1}z^2$.\end{proof}

\begin{lem}\label{lem:epsiloncommut}Let $F_4$ be the free group over $x,y,z,t$ and $\varepsilon\colon F_4\to\{1,-1\}$ the morphism defined by
$$\varepsilon(x)=\varepsilon(y)=1\qquad\text{and}\qquad\varepsilon(z)=\varepsilon(t)=-1.$$
There exist elements $x',y',z',t'\in F_4$ such that
$$\varepsilon(x')=\varepsilon(y')=\varepsilon(z')=\varepsilon(t')=-1\quad\text{and}\quad[x,y][z,t]=[x',y'][z',t']$$
and $F_4$ is the free group over $t',u',v',w'$.\end{lem}

\begin{proof}Take for instance
$$x'=xyz,\quad y'=z^{-1}x^{-1},\quad z'=(y^{-1}z)^{-1}z(y^{-1}z)\quad\text{and}\quad t'=tz^{-1}(y^{-1}z).$$
\end{proof}

\begin{lem}\label{lem:epsilonsquar}Let $F_4$ be the free group over $t,u,v,w$ and $\varepsilon\colon F_4\to\{1,-1\}$ the morphism defined by
$$\varepsilon(t)=\varepsilon(u)=\varepsilon(v)=1\qquad\text{and}\qquad\varepsilon(w)=-1.$$
There exist elements $t',u',v',w'\in F_4$ such that
$$\varepsilon(t')=1,\quad \varepsilon(u')=\varepsilon(v')=\varepsilon(w')=-1\quad\text{and}\quad t^2u^2v^2w^2=t'^2u'^2v'^2w'^2$$
and $F_4$ is the free group over $x',y',z',t'$.\end{lem}

\begin{proof}Take for instance $t',u',w',v'$ successively defined by:
$$t'=tu^2vu^{-1},\quad t'u'=u^2vw\quad u'w'=uvw^2\quad\text{and}\quad u'v'w'=w.$$\end{proof}

Using Lemmas \ref{lem:commutsquar}, \ref{lem:epsiloncommut} and \ref{lem:epsilonsquar}, we deduce from Lemma  \ref{lem:hybrid}:

\begin{prop}\label{prop:s1sm}If $\varphi$ maps $s_1,\ldots,s_m$ to $1$ and all other generators to $0$ then its kernel is the fundamental group of the Seifert manifold given by
$$\{e'; (\epsilon', G);F_m\}$$
where
$$e'=\begin{cases}2e&\text{if $\epsilon=o_1,n_2$}\\0&\text{if $\epsilon=o_2,n_1,n_3,n_4$}\end{cases}\qquad\epsilon'=\begin{cases}\epsilon&\text{if $\epsilon=o_1,o_2,n_1,n_2,n_4$}\\n_4&\text{if $\epsilon=n_3$}\end{cases}$$
$$G=\begin{cases}\frac m2-1+2g&\text{if $\epsilon=o_1,o_2$}\\m-2+2g&\text{if $\epsilon=n_1,n_2,n_3,n_4$}\end{cases}$$
and $F_m$ as defined in Notation \ref{nota:FOCFm}.
\end{prop}

\begin{proof}When $\epsilon=o_1$, Lemma \ref{lem:hybrid} directly gives the result and $\epsilon'=o_1$, with $2G=g''=m-2+2g'=m-2+4g$. When $\epsilon=o_2$, the number of $\varepsilon'_j$'s equal to $-1$ is $g'=2g>0$ hence by Lemma \ref{lem:epsiloncommut}, all $\varepsilon'_j$'s can be replaced by $-1$, so $\epsilon'=o_2$ and $2G=m-2+4g$ as in the previous case. When $\epsilon=n_1$, Lemma \ref{lem:commutsquar} may be applied if necessary (i.e. if $m>2$) to replace each commutator by a product of two squares (only a weak form of the Lemma is used here, forgetting about the morphism $\varepsilon$ of its statement). Thus we get $\epsilon'=n_1$ and $G=g''=m-2+2g'=m-2+2g$. When $\epsilon=n_2$, applying again Lemma \ref{lem:commutsquar} if necessary (in its strong form) gives the result and $\epsilon'=n_2$ with $G=m-2+2g$ as in the previous case. When $\epsilon=n_3$, $W=[v'_1,v'_2]\ldots[v'_{m-3},v'_{m-2}]v_{m-1}'^2\ldots v_{m-2+2g}'^2$ and among $\varepsilon'_{m-1},\ldots,\varepsilon'_{m-2+2g}$, there are two $1$'s, but the number of $-1$'s is $2g-2>0$ and we may reorder these $2g$ $\varepsilon'_j$'s to put the two $1$'s first (by a repeated change of variable, using that $u^2v^2=(u^2vu^{-2})^2u^2$). Lemma \ref{lem:commutsquar} again gives the conclusion. When $\epsilon=n_4$ the same method allows to transform $W$ into a product of squares but we are left with four $1$'s instead of two. Reordering again to put these four $1$'s at the beginning, Lemma \ref{lem:epsilonsquar} allows to reduce this number of $1$'s to two, hence (again) $\epsilon'=n_4$ (and $G=m-2+2g$).\end{proof}

The next lemma will be used in Theorem \ref{theo:s1smvj} to extend the result of the previous proposition to the general case, where $\varphi$ does not necessarily vanish on all $v_j$'s.

\begin{lem}\label{lem:killvj}If two morphisms from $\pi_1(N)$ to $\Z_2$ map $s_1,\ldots,s_m$ ($m>0$) to $1$ and $h,s_{m+1},\ldots,s_n$ to $0$ then their kernels are isomorphic.
\end{lem}

\begin{proof}~

When $\epsilon=n_1,n_2,n_3,n_4$, the idea is to change the presentation of $\pi_1(N)$, using $\varphi(s_1)=1$ to ``kill'' all the $\varphi(v_j)$'s, one after another. The formulas are simpler if we prepare each of these ``murders'' by temporarily permuting the $\varphi(v_j)$s', to put at the end the one whose value we want to switch from $1$ to $0$.

The values $\varphi(v_j)$ and $\varphi(v_{j+1})$ can be exchanged by the following change of generators (leaving the other generators untouched):$$v'_j=v_j^2v_{j+1}v_j^{-2}, v'_{j+1}=v_j.$$

Combining such transpositions, we can {\it temporarily} reorder the $v_j$'s (it may affect the conventional ordering of the $\varepsilon_j$'s when $\epsilon=n_3,n_4$, but this is temporary hence harmless).

Then, the value of $\varphi(v_g)$ can be switched from $1$ to $0$ by:
$$v'_g=v_gs_1,\qquad s'_1=v_g'^{-1}s_1^{-1}v'_g.$$
When $v_g$ anticommutes with $h$, this is in fact (like the previous one) an automorphism of $\pi_1(N)$, but when $v_g$ commutes with $h$, it is only a change of presentation, since $b_1$ is changed into its opposite. Thus, after ``killing'' all $\varphi(v_j)$'s and then applying Proposition \ref{prop:s1sm}, we get the same invariants for $\Ker(\varphi)$ as if all $\varphi(v_j)$'s were already $0$, up to a possible change of sign of $b_1$ in the type $(a_1/2,b_1)$  of the first exceptional fiber. However, this may only happen when some $\varepsilon_j$'s are equal to $1$, i.e. when $\epsilon=n_1,n_3,n_4$. But we see from Proposition \ref{prop:s1sm} that $\Ker(\varphi)$ inherits this non-orientability, which allows to replace in the final result such an $(a_1/2,-b_1)$ (if it occurs) by $(a_1/2,b_1)$.

When $\epsilon=o_1,o_2$, the idea is the same:\\
the value of $(\varphi(v_{2i-1}),\varphi(v_{2i}))$ can be exchanged with that of $(\varphi(v_{2i+1}),\varphi(v_{2i+2}))$ by: $$v'_{2i-1}=[v_{2i-1},v_{2i}]v_{2i+1}[v_{2i-1},v_{2i}]^{-1},\qquad v'_{2i}=[v_{2i-1},v_{2i}]v_{2i+2}[v_{2i+2},v_{2i}]^{-1},$$$$v'_{2i+1}=v_{2i-1},\qquad v'_{2i+2}=v_{2i},$$
and the values of $\varphi(v_j)$ and $\varphi(v_{j+1})$ when $j$ is odd can be exchanged by:
$$v'_j=v_jv_{j+1}v_j^{-1},\qquad v'_{j+1}=v_j^{-1}.$$
The value of $\varphi(v_{2g})$ can be switched from $1$ to $0$ by:
$$s'_1=v_{2g-1}s_1v_{2g-1}^{-1},\qquad v'_{2g-1}=s'_1v_{2g-1}s_1'^{-1},\qquad v'_{2g}=s'_1s_1^{-1}v_{2g}s_1'^{-1}.$$
When $\epsilon=o_2$, this again changes $b_1$ to $-b_1$ but it can be cured if necessary in the final result, by the same argument.

\end{proof}

From Proposition \ref{prop:s1sm} and Lemma \ref{lem:killvj} we immediately deduce:
\begin{theo}\label{theo:s1smvj}If $\varphi$  map $s_1,\ldots,s_m$ ($m>0$) to $1$ and $h,s_{m+1},\ldots,s_n$ to $0$ then its kernel is the fundamental group of the Seifert manifold given by
$$\{e'; (\epsilon', G);F_m\}$$
where
$$e'=\begin{cases}2e&\text{if $\epsilon=o_1,n_2$}\\0&\text{if $\epsilon=o_2,n_1,n_3,n_4$}\end{cases}\qquad\epsilon'=\begin{cases}\epsilon&\text{if $\epsilon=o_1,o_2,n_1,n_2,n_4$}\\n_4&\text{if $\epsilon=n_3$}\end{cases}$$
$$G=\begin{cases}\frac m2-1+2g&\text{if $\epsilon=o_1,o_2$}\\m-2+2g&\text{if $\epsilon=n_1,n_2,n_3,n_4$}\end{cases}$$
and $F_m$ as defined in Notation \ref{nota:FOCFm}.
\end{theo}

\subsubsection{If $\varphi$ maps some $v_j$'s to $1$ and all other generators $0$}\label{vj}

In this subsection, $\varphi(h)=\varphi(s_1)=\ldots=\varphi(s_n)=0$. Apart from orientation covers and some ``exotic'' cases, $\Ker(\varphi)$ will have the same description as in Theorem \ref{theo:s1smvj}, with $m$ replaced by $0$. The two cases $\epsilon=o_1,o_2$ and $\epsilon=n_1,n_2,n_3,n_4$ will be treated separately (Propositions \ref{prop:vjoi} and \ref{prop:vjni}) and the global result will be rephrased in Theorem \ref{theo:vj}.

\begin{prop}\label{prop:vjoi}When $\varphi$ maps $r>0$  generators $v_j$'s to $1$ and all other generators to $0$ and $\epsilon=o_1$ or $o_2$, $\Ker(\varphi)$ is the fundamental group of the Seifert manifold given by the following invariants, whith $F_{OC}$ and  $F_0$ as defined in Notation \ref{nota:FOCFm}.
\begin{itemize}
\item if $\epsilon=o_1$: $\{2e; (o_1,2g-1);F_0\}$
\item if $\epsilon=o_2$ and $r=g$ (orientation cover): $\{0;(o_1, 2g-1);F_{OC}\}$
\item in all other cases of $\epsilon=o_2$: $\{0;(o_2,2g-1);F_0\}$.
\end{itemize}
\end{prop}

\begin{proof}Let us reorder $r$ the number of $v_j$'s mapped to 1  in such a way that $\varphi(v_j)=1\Leftrightarrow j\le r$  (by the same method as in Lemma \ref{lem:killvj}). Let $\varepsilon$ be the common value of the $\varepsilon_j$'s, i.e. $\varepsilon=1$ if $\epsilon=o_1$ and to $\varepsilon=-1$ if $\epsilon=o_2$.

Choosing $q=v_r$, Reidemeister-Schreier's algorithm produces a presentation of $\Ker(\varphi)$ with

\begin{itemize}
\item generators:
	\begin{itemize}
	\item for $1\le k\le n$, $(y_k,y'_k)=(s_k,qs_kq^{-1})$
	\item for $1\le j\le 2g$, $$(x_j,x'_j)=\begin{cases}(v_jq^{-1},qv_j)&\text{if $j\le r$}\\(v_j,qv_jq^{-1})&\text{if $j>r$}\end{cases}$$
	\item $(z,z')=(h,qhq^{-1})$
	\end{itemize}
\item relations:
	\begin{itemize}
	\item $x_r=1$, $z'=z^\varepsilon$
	\item $z$ commutes with all $y_k$'s and $y'_k$'s, and with $x_j$ and $x'_j$ for $j\le r$
	\item $x_jzx_j^{-1}z^{-\varepsilon}=x'_jzx_j'^{-1}z^{-\varepsilon}=1$ for $j>r$
	\item $y_k^{a_k}z^{b_k}=y_k'^{a_k}z^{\varepsilon b_k}=1$ ($\forall k=1,\ldots,n$)
	\item (I) $YB=z^e$, (II) $Y'B'=z^{\varepsilon_re}$, where  $Y=y_1\ldots y_n$, $Y'=y'_1\ldots y'_n$, and $B,B'$ are described as follows:
	\end{itemize}
\item if $r$ is odd, $B=Zx'_{r+1}x_{r+1}^{-1}X$ and $B'=Z'x'_rx_{r+1}x_r'^{-1}x_{r+1}'^{-1}X'$, where
$$X=[x_{r+2},x_{r+3}]\ldots[x_{2g-1},x_{2g}],\quad X'=[x'_{r+2},x'_{r+3}]\ldots[x'_{2g-1},x'_{2g}],$$ $$Z=(x_1x'_2x_1'^{-1}x_2^{-1})\ldots(x_{r-2}x'_{r-1}x_{r-2}'^{-1}x_{r-1}^{-1}),$$
$$Z'=(x'_1x_2x_1^{-1}x_2'^{-1})\ldots(x'_{r-2}x_{r-1}x_{r-2}^{-1}x_{r-1}'^{-1});$$
\item if $r$ is even, $B=Zx_{r-1}x'_rx_{r-1}'^{-1}X$ and $B'=Z'x'_{r-1}x_{r-1}^{-1}x_r'^{-1}X'$, where $$X=[x_{r+1},x_{r+2}]\ldots[x_{2g-1},x_{2g}],\quad X'=[x'_{r+1},x'_{r+2}]\ldots[x'_{2g-1},x'_{2g}],$$ $$Z=(x_1x'_2x_1'^{-1}x_2^{-1})\ldots(x_{r-3}x'_{r-2}x_{r-3}'^{-1}x_{r-2}^{-1}),$$
$$Z'=(x'_1x_2x_1^{-1}x_2'^{-1})\ldots(x'_{r-3}x_{r-2}x_{r-3}^{-1}x_{r-2}'^{-1}).$$
\end{itemize}

Assume first that $r$ is odd. Using (II) to eliminate $x'_{r+1}$ in (I), and replacing some of the generators by conjugates thereof (without altering the previous relations) to reorder the subexpressions of the resulting relation, (I) becomes:
$$YY'W=z^{(1+\varepsilon)e},\qquad\text{with}\qquad W=ZZ'[x'_r,x_{r+1}]X'X.$$
Transforming $ZZ'$ by lemma \ref{lem:ZZ'} below (which is stated informally but whose proof gives explicit formulas), $W$ becomes a product of $(r-1)+1+(2g-r-1)$ commutators of new generators, whose associated $\varepsilon'_j$'s are equal to $1$ for the first $2(r-1)+1$ of them and to $\varepsilon$ for the $1+2(2g-r-1)$ last ones. We thus get the following Seifert invariants for $\Ker(\varphi)$:
\begin{itemize}
\item if $\epsilon=o_1$: $\{2e; (o_1,2g-1);F_0\}$
\item if $\epsilon=o_2$ (noting that $1+2(2g-r-1)>0$): $\{0;(o_2, 2g-1);F_{OC}\}$.
\end{itemize}
In the $o_2$ (non-orientable) case, all $-b_k$'s can be replaced by their opposites, i.e. $F_{OC}$ replaced by $F_0$, which concludes the odd case.

Assume now that $r$ is even. Using (I) to eliminate $x'_{r-1}$ and performing similar transformations, (II) becomes:
$$Y'YW=z^{(1+\varepsilon)e},\qquad\text{with}\qquad W=Z'Z[x_{r-1},x'_r]XX'.$$
Using lemma \ref{lem:ZZ'} again, $W$ becomes a product $(r-2)+1+(2g-r)$ commutators and the first $2(r-2)+2$  $\varepsilon'_j$'s are equal to $1$, the $2(2g-r)$ last ones being equal to $\varepsilon$. Hence the conclusion is the same as in the odd case, except when $\varepsilon=-1$ and $2(2g-r)=0$, which corresponds to the orientation cover case of the statement. This concludes the proof of Proposition \ref{prop:vjoi}, provided the next Lemma.
\end{proof}

\begin{lem}\label{lem:ZZ'}In any group, an expression of the form
$$(a_1b_1c_1d_1)\ldots(a_kb_kc_kd_k)(c_1^{-1}d_1^{-1}a_1^{-1}b_1^{-1})\ldots(c_k^{-1}d_k^{-1}a_k^{-1}b_k^{-1})$$
is the product of $2k$ commutators.\end{lem}

\begin{proof} Let $U_k=(a_1b_1c_1d_1)\ldots(a_kb_kc_kd_k)$ and $V_k=(c_1^{-1}d_1^{-1}a_1^{-1}b_1^{-1})\ldots(c_k^{-1}d_k^{-1}a_k^{-1}b_k^{-1})$. Then $U_0V_0=1$, and $U_{k+1}V_{k+1}$ is the product of $U_kV_k$ (which by induction hypothesis is a product of $2k$ commutators) by
$$V_k^{-1}a_{k+1}b_{k+1}c_{k+1}d_{k+1}V_kc_{k+1}^{-1}d_{k+1}^{-1}a_{k+1}^{-1}b_{k+1}^{-1}=$$$$[V_k^{-1}a_{k+1},b_{k+1}V_k](b_{k+1}a_{k+1})[V_k^{-1}c_{k+1},d_{k+1}V_k](b_{k+1}a_{k+1})^{-1}.$$
\end{proof}

Proposition \ref{prop:vjoi} dealt with the case $\epsilon=o_1,o_2$. The next proposition deals with the other case, $\epsilon=n_1,n_2,n_3,n_4$.

\begin{prop}\label{prop:vjni}When $\varphi$ maps some $v_j$'s to $1$ and all other generators to $0$ and $\epsilon=n_1,n_2,n_3$ or $n_4$, $\Ker(\varphi)$ is the fundamental group of the Seifert manifold given by the following invariants, whith $F_{OC}$ and $F_0$ as defined in Notation \ref{nota:FOCFm}.
\begin{itemize}
\item (Orientation covers)
	\begin{itemize}
	\item  if $\epsilon=n_1$ and $\varphi$ maps all $v_j$'s to $1$: $\{0; (o_1, g-1);F_{OC}\}$
	\item if $\epsilon=n_3$ and $\varphi$ sends only $v_1$ to $1$, or if $\epsilon=n_4$ and $\varphi$ sends only $v_1,v_2$ to $1$: $\{0;(n_2,2g-2);F_{OC}\}$
	\end{itemize}
\item (Exotic cases) if $\varphi$ maps all $v_j$'s to $1$ but $\epsilon\ne n_1$
	\begin{itemize}
	\item if $\epsilon=n_2$: $\{2e; (o_1, g-1);F_0\}$
	\item if $\epsilon=n_3,n_4$: $\{0; (o_2, g-1);F_0\}$
	\end{itemize}
\item (Ordinary cases)  in all other cases: $\{e';(\epsilon',2g-2);F_0\}$ with
$$\epsilon'=\begin{cases}\epsilon&\text{if $\epsilon=n_1,n_2,n_4$}\\n_4&\text{if $\epsilon=n_3$}\end{cases}\qquad\text{and}\qquad e'=\begin{cases}2e&\text{if $\epsilon=n_2$}\\0&\text{if $\epsilon=n_1,n_3,n_4$.}\end{cases}$$
\end{itemize}
\end{prop}

\begin{proof}Assume, like in the proof of Proposition \ref{prop:vjoi}, that $\varphi(v_j)=1\Leftrightarrow j\le r$. We shall have to be cautious about a new phenomenon: this reordering of the $v_j$'s may affect the ordering of the $\varepsilon_j$'s, i.e. we shall try to maintain the convention that the $\varepsilon_j$'s equal to 1 are allways $\varepsilon_1$ when $\epsilon=n_3$ and $\varepsilon_1,\varepsilon_2$ when $\epsilon=n_4$, but we shall locally drop this convention inside the present proof whenever it is not compatible with our reordering of the $v_j$'s.

Choosing $q=v_1$, Reidemeister-Schreier's algorithm produces a presentation of $\Ker(\varphi)$ with
\begin{itemize}
\item generators:
	\begin{itemize}
	\item for $1\le k\le n$, $(y_k,y'_k)=(s_k,qs_kq^{-1})$
	\item for $1\le j\le g$, $$(x_j,x'_j)=\begin{cases}(v_jq^{-1},qv_j)&\text{if $j\le r$}\\(v_j,qv_jq^{-1})&\text{if $j>r$}\end{cases}$$
	\item $(z,z')=(h,qhq^{-1})$
	\end{itemize}
\item relations:
	\begin{itemize}
	\item $x_1=1$, $z'=z^{\varepsilon_1}$
	\item $z$ commutes with all $y_k$'s and $y'_k$'s
	\item $x_jzx_j^{-1}z^{-\varepsilon'_j}=x'_jzx_j'^{-1}z^{-\varepsilon'_j}=1$ whith $\varepsilon'_j=\begin{cases}\varepsilon_j\varepsilon_1&\text{if $j\le r$}\\\varepsilon_j&\text{if $j>r$}\end{cases}$
	\item $y_k^{a_k}z^{b_k}=y_k'^{a_k}z^{\varepsilon_1 b_k}=1$ ($\forall k=1,\ldots,n$)
	\item (I) $Yx'_1ZX=z^e$, where $Y=y_1\ldots y_n$, $X=x_{r+1}^2\ldots x_n^2$ and
$Z=x_2x'_2\ldots x_rx'_r$
	\item (II) $Y'x'_1Z'X'=z^{\varepsilon_1e}$, where $Y'=y'_1\ldots y'_n$, $X'=x'_{r+1}\ldots x'_n$ and $Z'=x'_2x_2\ldots x'_rx_r$.
	\end{itemize}
\end{itemize}
Eliminating $x'_1$, (I) and (II) join to become (III):
$YY'^{-1}X'^{-1}Z'^{-1}ZX=z^{(1-\varepsilon_1)e}$ and $Z'^{-1}Z$ is a product of $r-1$ commutators (of conjugates of inverses of $x_2,x'_2,\ldots,x_r,x'_r$, having the same $\varepsilon'_j$'s).

When $r=g$, $X$ and $X'$ are empty products hence $\epsilon'=o_1$ or $o_2$. Moreover when $\epsilon=n_3,n_4$, since we find $\epsilon'=o_2$, we can replace the $-b_k$'s which occur by their opposites. More precisely Seifert invariants for $\Ker(\varphi)$ when $r=g$ are:
\begin{itemize}
\item if $\epsilon=n_1$: $\{0; (o_1, g-1);F_{OC}\}$
\item if $\epsilon=n_2$: $\{2e; (o_1, g-1);F_0\}$
\item if $\epsilon=n_3,n_4$: $\{0; (o_2, g-1);F_0\}$.
\end{itemize}

When $r<g$, (III) contains a product of $2(g-r)$ squares and $r-1$ commutators, which can be converted to a product of $2g-2$ squares (using Lemma \ref{lem:commutsquar} and taking care of the $\varepsilon'_j$'s to determine $\epsilon'$). Moreover, when the $\epsilon'$ we find corresponds to a non-orientable manifold, all $b'_k=-b_k$'s (if any) can be replaced by $b'_k=b_k$. Hence Seifert invariants for $\Ker(\varphi)$ when $r<g$ are:
\begin{itemize}
\item if $\epsilon=n_2$: $\{2e;n_2,2g-2);F_0\}$
\item if $\epsilon=n_3$ and $\varphi$ sends only $v_1$ to $1$, or if $\epsilon=n_4$ and $\varphi$ sends only $v_1,v_2$ to $1$:
$\{0;(n_2,2g-2);F_{OC}\}$
\item in all other cases: $\{0;\epsilon',2g-2);F_0\}$
with $\epsilon'=n_1$ if $\epsilon=n_1$, and $\epsilon'=n_4$ if $\epsilon=n_3,n_4$.
\end{itemize}
\end{proof}

The following theorem is a synthesis of Propositions \ref{prop:vjoi} and \ref{prop:vjni}.

\begin{theo}\label{theo:vj}When $\varphi$ maps some $v_j$'s to $1$ and all other generators to $0$, $\Ker(\varphi)$ is the fundamental group of the Seifert manifold given by the following invariants, whith $F_{OC}$ and and $F_0$ as defined in Notation \ref{nota:FOCFm}.
\begin{itemize}
\item (Orientation covers)
	\begin{itemize}
	\item if $\epsilon=o_2$ and $\varphi$ maps all $v_j$'s to $1$: $\{0;(o_1,2g-1);F_{OC}\}$
	\item if $\epsilon=n_1$ and $\varphi$ maps all $v_j$'s to $1$: $\{0;(o_1,g-1);F_{OC}\}$
	\item if $\epsilon=n_3$ and $\varphi$ sends only $v_1$ to $1$, or if $\epsilon=n_4$ and $\varphi$ sends only $v_1,v_2$ to $1$: $\{0;(n_2,2g-2);F_{OC}\}$
	\end{itemize}
\item (Exotic cases) if $\epsilon=n_2,n_3,n_4$ and $\varphi$ maps all $v_j$'s to $1$:
	\begin{itemize}
	\item if $\epsilon=n_2$: $\{2e; (o_1, g-1);F_0\}$
	\item if $\epsilon=n_3,n_4$: $\{0; (o_2, g-1);F_0\}$
	\end{itemize}
\item (Ordinary cases)  in all other cases: $\{e';(\epsilon',G);F_0\}$ with
$$e'=\begin{cases}2e&\text{if $\epsilon=o_1,n_2$}\\0&\text{if $\epsilon=o_2,n_1,n_3,n_4$}\end{cases}\qquad\epsilon'=\begin{cases}\epsilon&\text{if $\epsilon=o_1,o_2,n_1,n_2,n_4$}\\n_4&\text{if $\epsilon=n_3$}\end{cases}$$$$\text{and}~G=\begin{cases}2g-1&\text{if $\epsilon=o_1,o_2$}\\2g-2&\text{if $\epsilon=n_1,n_2,n_3,n_4$.}\end{cases}.$$
\end{itemize}
\end{theo}

\author{Anne Bauval}\\
\address{\small Institut de Math\'ematiques de Toulouse\\
Equipe Emile Picard, UMR 5580\\
Universit\'e Toulouse III\\
118 Route de Narbonne, 31400 Toulouse - France\\
e-mail: bauval@math.univ-toulouse.fr}

\author{Claude Hayat}\\
\address{\small Institut de Math\'ematiques de Toulouse\\
Equipe Emile Picard, UMR 5580\\
118 Route de Narbonne, 31400 Toulouse - France\\
e-mail: hayat@math.univ-toulouse.fr}

\end{document}